\documentclass{amsart}
\usepackage{amsfonts, amssymb, amsmath, eucal, verbatim, amsthm, amscd, enumerate}
\usepackage{graphicx}
\newtheorem{theorem}{Theorem}[section]
\newtheorem{lemma}[theorem]{Lemma}

\newtheorem*{rems}{Remarks}
\theoremstyle{definition}

\theoremstyle{remark}

\numberwithin{equation}{section} \oddsidemargin .5in \evensidemargin
.5in \textwidth=5.5in
\def\one{{\mathbf{1}}}

\begin{document}
\title{Heat-flow monotonicity of Strichartz norms}
\author{Jonathan Bennett}
\author{Neal Bez}
\author{Anthony Carbery}
\author{Dirk Hundertmark}
\thanks{The first and second authors were supported by EPSRC grant EP/E022340/1}
\address{Jonathan Bennett and Neal Bez \\ School of Mathematics \\
The University of Birmingham \\
The Watson Building \\
Edgbaston \\
Birmingham \\
B15 2TT \\
United Kingdom}\email{J.Bennett@bham.ac.uk\\ N.Bez@bham.ac.uk}
\address{Anthony Carbery \\ School of Mathematics and Maxwell Institute for Mathematical
Sciences \\ The University of Edinburgh \\ James Clerk Maxwell
Building \\ King's Buildings \\ Edinburgh \\ EH3 9JZ \\ United
Kingdom}\email{A.Carbery@ed.ac.uk}
\address{Dirk Hundertmark \\ Department of Mathematics\\
University of Illinois at Urbana--Champaign\\
Urbana\\ Illinois 61801\\ USA}\email{dirk@math.uiuc.edu}
\subjclass[2000]{35Q40, 35K05} \keywords{Heat-flow, Strichartz
estimates, Schr\"odinger equation}
\date{26th of September 2008}
%\dedicatory{}
%\commby{}
\begin{abstract} Most notably we prove that for $d=1,2$ the classical Strichartz norm
$$\|e^{i s\Delta}f\|_{L^{2+4/d}_{s,x}(\mathbb{R}\times\mathbb{R}^d)}$$
associated to the free Schr\"{o}dinger equation is nondecreasing as
the initial datum $f$ evolves under a certain quadratic heat-flow.
\end{abstract}

\maketitle

\section{Introduction}
For $d\in\mathbb{N}$ let the Fourier transform
$\widehat{f}:\mathbb{R}^d\rightarrow\mathbb{C}$ of a Lebesgue
integrable function $f$ on $\mathbb{R}^d$ be given by
\begin{equation*}
\widehat{f}(\xi) = \frac{1}{(2\pi)^{d/2}}\int_{\mathbb{R}^d}e^{-i
x\cdot\xi}f(x)\,dx.
\end{equation*}
For each $s\in\mathbb{R}$ the Fourier multiplier operator $e^{i
s\Delta}$ is defined via the Fourier transform by
\begin{equation*}
\widehat{e^{i s\Delta}f}(\xi) = e^{-i s|\xi|^2}\widehat{f}(\xi),
\end{equation*}
for all $f$ belonging to the Schwartz class
$\mathcal{S}(\mathbb{R}^d)$ and $\xi\in\mathbb{R}^d$. Thus for each
$f \in \mathcal{S}(\mathbb{R}^d)$ and $x\in\mathbb{R}^d$,
\begin{equation*}
e^{i s\Delta}f(x) = \frac{1}{(2\pi)^{d/2}} \int_{\mathbb{R}^d}
e^{i(x \cdot \xi - s|\xi|^2)}\widehat{f}(\xi)\,d\xi.
\end{equation*}
By an application of the Fourier transform in $x$ it is easily seen
that $e^{i s\Delta}f(x)$ solves the Schr\"odinger equation
\begin{equation} \label{e:Schrodinger}
i\partial_s u = -\Delta u,
\end{equation}
with initial datum $u(0,x)=f(x)$. It is well known that the above
solution operator $e^{i s\Delta}$ extends to a bounded operator from
$L^2(\mathbb{R}^d)$ to $L^{p}_sL^q_x(\mathbb{R} \times
\mathbb{R}^d)$ if and only if $(d,p,q)$ is Schr\"odinger admissible;
i.e. there exists a finite constant $C_{p,q}$ such that
\begin{equation} \label{e:Strichartz}
  \|e^{i s\Delta}f\|_{L^{p}_sL^q_x(\mathbb{R} \times \mathbb{R}^d)} \leq
  C_{p,q}\|f\|_{L^2(\mathbb{R}^d)}
\end{equation}
if and only if
\begin{equation}\label{scal}
\text{$p,q \geq 2$, \,\, $(d,p,q) \neq (2,2,\infty)$ \,\, and \,\,
$\frac{2}{p}+\frac{d}{q}=\frac{d}{2}$.}
\end{equation}
For $p=q=2+4/d$, this classical inequality is due to Strichartz
\cite{Strichartz} who followed arguments of Stein and Tomas
\cite{Tomas}. For $p\not=q$ the reader is referred to \cite{KT} for
historical references and a full treatment of \eqref{e:Strichartz}
for suboptimal constants $C_{p,q}$.

Recently Foschi \cite{Foschi} and independently Hundertmark and
Zharnitsky \cite{HZ} showed that in the cases where one can
``multiply out" the Strichartz norm \begin{equation}\label{strinorm}
\|e^{i s\Delta}f\|_{L^{p}_sL^q_x(\mathbb{R} \times \mathbb{R}^d)},
\end{equation}
that is, when $q$ is an even integer which divides $p$, the sharp
constants $C_{p,q}$ in the above inequalities are obtained by
testing on isotropic centred gaussians. (These authors considered
$p=q$ only.) The main purpose of this paper is to highlight a
startling monotonicity property of such Strichartz norms as the
function $f$ evolves under a certain quadratic heat-flow.

\begin{theorem}\label{t:main} Let $f  \in L^2(\mathbb{R}^d)$. If $(d,p,q)$ is
Schr\"odinger admissible and $q$ is an even integer which divides
$p$ then the quantity
\begin{equation} \label{e:Qpq}
  Q_{p,q}(t) := \|e^{i s\Delta}(e^{t\Delta}|f|^2)^{1/2}
  \|_{L^p_sL^q_x(\mathbb{R} \times \mathbb{R}^d)}
\end{equation}
is nondecreasing for all $t>0$; i.e. $Q_{p,q}$ is nondecreasing in
the cases $(1,6,6)$, $(1,8,4)$ and $(2,4,4)$.
\end{theorem}
The heat operator $e^{t\Delta}$ is of course defined to be the
Fourier multiplier operator with multiplier $e^{-t|\xi|^2}$, and so
$$e^{t\Delta}|f|^2=H_t*|f|^2,$$
where the heat kernel $H_t:\mathbb{R}^d\rightarrow\mathbb{R}$ is
given by
\begin{equation} \label{e:heatkernel}
H_t(x) = \frac{1}{(4\pi t)^{d/2}} \, e^{-|x|^2/4t}.
\end{equation}

By making an appropriate rescaling one may rephrase the above result
in terms of ``sliding gaussians" in the following way. For $f\in
L^2(\mathbb{R}^d)$ let $u : (0,\infty) \times \mathbb{R}^d
\rightarrow \mathbb{R}$ be given by $u(t,x) = H_t*|f|^2(x)$ and
$\widetilde{u} : (0,\infty) \times \mathbb{R}^d \rightarrow
\mathbb{R}$ be given by
\begin{equation*}
\widetilde{u}(t,x) = t^{-d}u(t^{-2},t^{-1}x) =
\frac{1}{(4\pi)^{d/2}}
\int_{\mathbb{R}^d}e^{-\frac{1}{4}|x-tv|^2}|f(v)|^2\,dv.
\end{equation*}
We interpret $\widetilde{u}$ as a superposition of translates of a
fixed gaussian which simultaneously slide to the origin as $t$ tends
to zero. By a simple change of variables it follows that
\begin{equation} \label{e:rescaledQ}
Q_{p,q}(t^{-2})=\|e^{i
s\Delta}(\widetilde{u}(t,\cdot)^{1/2})\|_{L^p_sL^q_x(\mathbb{R}
\times \mathbb{R}^d)}.
\end{equation}
The reader familiar with the standard wave-packet analysis in the
context of Fourier extension estimates may find it more enlightening
to interpret Theorem \ref{t:main} via this rescaling.

The claimed monotonicity of $Q_{p,q}$ yields the sharp constant
$C_{p,q}$ in \eqref{e:Strichartz} as a simple corollary. To see
this, suppose that the function $f$ is bounded and has compact
support. Then, by rudimentary calculations,
\begin{equation*}
\lim_{t \rightarrow 0} Q_{p,q}(t) =
\|\,e^{is\Delta}|f|\,\|_{L^{p}_sL^q_x(\mathbb{R} \times
\mathbb{R}^d)}
\end{equation*}
which, by virtue of the fact that $q$ is an even integer which
divides $p$, is greater than or equal to $\|e^{i
s\Delta}f\|_{L^{p}_s(L^q_x(\mathbb{R}^d))}$. Furthermore, because of
\eqref{e:rescaledQ} it follows that
\begin{equation*}
\lim_{t \rightarrow \infty} Q_{p,q}(t) = \|e^{i
s\Delta}(H_1^{1/2})\|_{L^{p}_sL^q_x(\mathbb{R} \times
\mathbb{R}^d)}\|f\|_{L^2(\mathbb{R}^d)},
\end{equation*}
where $H_1$ is the heat kernel at time $t=1$. Therefore Theorem
\ref{t:main} gives the sharp constant $C_{p,q}$ in
\eqref{e:Strichartz} for the triples $(1,6,6)$, $(1,8,4)$ and
$(2,4,4)$, and shows that gaussians are maximisers. In particular,
if
\begin{equation*}
    C_{p,q} := \sup\{ \|e^{i s\Delta}f\|_{L^{p}_sL^q_x(\mathbb{R} \times \mathbb{R}^d)}
    : \|f\|_{L^2(\mathbb{R}^d)} = 1\}
\end{equation*}
then $C_{6,6} = 12^{-1/12}$, $C_{8,4} = 2^{-1/4}$ and $C_{4,4} =
2^{-1/2}$. As we have already noted, $C_{6,6}$ and $C_{4,4}$ were
found recently by Foschi \cite{Foschi} and independently Hundertmark
and Zharnitsky \cite{HZ}. In the $(1,8,4)$ case, we shall see in the
proof of Theorem \ref{t:main} below that the monotonicity (and hence
sharp constant) follows in a cheap way from the $(2,4,4)$ case.

Heat-flow methods have already proved effective in treating certain
$d$-linear analogues of the Strichartz estimate
\eqref{e:Strichartz}; see Bennett, Carbery and Tao \cite{BCT}. Also
intimately related (as we shall see) are the works of Carlen, Lieb,
and Loss \cite{CLL} and Bennett, Carbery, Christ and Tao \cite{BCCT}
in the setting of the multilinear Brascamp--Lieb inequalities.

We prove Theorem \ref{t:main} in Section 2, and discuss some further
results in Section 3.

\section{Proof of Theorem \ref{t:main}}

The idea behind the proof of Theorem \ref{t:main} is simply to
express the Strichartz norm
\begin{equation*}
\|e^{is\Delta}f\|_{L^{p}_sL^q_x(\mathbb{R} \times \mathbb{R}^d)}
\end{equation*}
in terms of quantities which are already known to be monotone under
the heat-flow that we consider. As we shall see, this essentially
amounts to bringing together the Strichartz-norm representation
formulae of Hundertmark and Zharnitsky \cite{HZ} and the following
heat-flow monotonicity property inherent in the Cauchy--Schwarz
inequality.

\begin{lemma} \label{l:CS} For $n \in \mathbb{N}$ and nonnegative integrable functions
$f_1$ and $f_2$ on $\mathbb{R}^n$ the quantity
\begin{equation*}
    \Lambda(t) := \int_{\mathbb{R}^n} (e^{t\Delta}f_1)^{1/2}(e^{t\Delta}f_2)^{1/2}
\end{equation*}
is nondecreasing for all $t>0$.
\end{lemma}

\begin{proof} Let $0 < t_1 < t_2$. If $H_t$ denotes the heat kernel
on $\mathbb{R}^n$ given by \eqref{e:heatkernel} then,
\begin{align*}
  \Lambda(t_1) & = \int_{\mathbb{R}^n} (H_{t_1} * f_1)^{1/2} (H_{t_1} *
  f_2)^{1/2} \\
  & = \int_{\mathbb{R}^n} H_{t_2-t_1} * ((H_{t_1} * f_1)^{1/2} (H_{t_1} *
  f_2)^{1/2}) \\
  & = \int_{\mathbb{R}^n} \int_{\mathbb{R}^n} (H_{t_2-t_1}(x-y)H_{t_1}*f_1(y))^{1/2}
  (H_{t_2-t_1}(x-y)H_{t_1}*f_2(y))^{1/2}\,dydx \\
  & \leq \int_{\mathbb{R}^n} (H_{t_2-t_1} * (H_{t_1} * f_1))^{1/2} (H_{t_2-t_1} * (H_{t_1} *
  f_2))^{1/2} \\
  & = \Lambda(t_2),
\end{align*}
where we have used the Cauchy--Schwarz inequality on
$L^2(\mathbb{R}^n)$ and the semigroup property of the heat kernel.
\end{proof}

The above proof of Lemma \ref{l:CS} originates in work of Ball
\cite{Ball} and was developed further in \cite{BCCT}. An alternative
method of proof in \cite{CLL} and \cite{BCCT} which is based on the
divergence theorem produces the explicit formula
  \begin{equation}\label{e:csform}
    \Lambda'(t) = \frac{1}{4}\int_{\mathbb{R}^n} |\nabla(\log e^{t\Delta}f_1)
    -\nabla(\log e^{t\Delta}f_2)|^2
    (e^{t\Delta}f_1)^{1/2}(e^{t\Delta}f_2)^{1/2}
  \end{equation}
for each $t > 0$ provided $f_1$ and $f_2$ are sufficiently
well-behaved (such as bounded with compact support). We remark in
passing that the Cauchy--Schwarz inequality on $L^2(\mathbb{R}^n)$
follows from Lemma \ref{l:CS} by comparing the limiting values of
$\Lambda(t)$ for $t$ at zero and infinity.

The next lemma is an observation of Hundertmark and Zharnitsky
\cite{HZ} who showed that multiplied out expressions for the
Strichartz norm in the $(1,6,6)$ and $(2,4,4)$ cases have a
particularly simple geometric interpretation.

\begin{lemma} \label{l:HZ}
\emph{(1)} For nonnegative $f \in L^2(\mathbb{R})$,
\begin{equation*}
\|e^{is\Delta}f\|_{L^6_sL^6_x(\mathbb{R} \times \mathbb{R})}^6 =
\frac{1}{2\sqrt{3}}\int_{\mathbb{R}^3} (f \otimes f \otimes f)(X)
P_1(f \otimes f \otimes f)(X)\,dX
\end{equation*}
where $P_1:L^2(\mathbb{R}^3) \rightarrow L^2(\mathbb{R}^3)$ is the
projection operator onto the subspace of functions on $\mathbb{R}^3$
which are invariant under the isometries which fix the direction
$(1,1,1)$.

\emph{(2)} For nonnegative $f \in L^2(\mathbb{R}^2)$,
\begin{equation*}
\|e^{is\Delta}f\|_{L^4_sL^4_x(\mathbb{R} \times\mathbb{R}^2)}^4 =
\frac{1}{4} \int_{\mathbb{R}^4}(f \otimes f)(X) P_2(f \otimes
f)(X)\,dX
\end{equation*}
where $P_2:L^2(\mathbb{R}^4) \rightarrow L^2(\mathbb{R}^4)$ is the
projection operator onto the subspace of functions on $\mathbb{R}^4$
which are invariant under the isometries which fix the directions
$(1,0,1,0)$ and $(0,1,0,1)$.
\end{lemma}

\begin{proof}[Proof of Theorem \ref{t:main}]
We begin with the case where $(p,q,d)$ is equal to $(1,6,6)$. For
functions $G \in L^2(\mathbb{R}^3)$ we may write
\begin{equation} \label{e:Prep}
  P_1G(X) = \int_{O} G(\rho X)\,d\mathcal{H}(\rho)
\end{equation}
where $O$ is the group of isometries on $\mathbb{R}^3$ which
coincide with the identity on the span of $(1,1,1)$ and
$d\mathcal{H}$ denotes the right-invariant Haar probability measure
on $O$.

If, for $f \in L^2(\mathbb{R})$, we let $F := f \otimes f \otimes f$
then it is easy to see that
\begin{equation} \label{e:tensor}
e^{t\Delta}|f|^2 \otimes e^{t\Delta}|f|^2 \otimes e^{t\Delta}|f|^2 =
e^{t\Delta}|F|^2
\end{equation}
because, in general, the heat operator $e^{t \Delta}$ commutes with
tensor products. It is also easy to check that for each isometry
$\rho$ on $\mathbb{R}^3$,
\begin{equation} \label{e:isom}
(e^{t\Delta}|f|^2 \otimes e^{t\Delta}|f|^2 \otimes e^{t\Delta}|f|^2)
(\rho \, \cdot) = e^{t\Delta}|F_\rho|^2
\end{equation}
where $F_\rho := F(\rho \, \cdot)$. In \eqref{e:tensor} and
\eqref{e:isom} the Laplacian $\Delta$ acts in the number of
variables dictated by context. Therefore, by Lemma \ref{l:HZ}(1),
\begin{equation*}
  Q_{6,6}(t)^6 = \frac{1}{2\sqrt{3}}  \int_O \int_{\mathbb{R}^3} (e^{t\Delta}|F|^2)
  ^{1/2}(X) (e^{t\Delta}|F_\rho|^2)^{1/2}(X) \,dXd\mathcal{H}(\rho)
\end{equation*}
and, by Lemma \ref{l:CS} and the nonnegativity of the measure
$d\mathcal{H}$, it follows that $Q_{6,6}(t)$ is nondecreasing for
each $t
> 0$.

For the $(2,4,4)$ case, we use a representation of the form
\eqref{e:Prep} for the projection operator $P_2$ where the averaging
group $O$ is replaced by the group of isometries on $\mathbb{R}^4$
which coincide with the identity on the span of $(1,0,1,0)$ and
$(0,1,0,1)$. Of course, the analogous statements to \eqref{e:tensor}
and \eqref{e:isom} involving two-fold tensor products hold. Hence
the nondecreasingness of $Q_{4,4}$ follows from Lemma \ref{l:HZ}(2)
and Lemma \ref{l:CS}.

Finally, for the $(1,8,4)$ case we observe that
\begin{equation*}
  \|e^{is \Delta}(e^{t \Delta} |f|^2)^{1/2}
  \|_{L^8_sL^4_x(\mathbb{R} \times \mathbb{R})}^2 = \|e^{is \Delta}(e^{t \Delta} (|f|^2 \otimes
  |f|^2))^{1/2} \|_{L^4_sL^4_x(\mathbb{R} \times\mathbb{R}^2)}
\end{equation*}
because both solution operators $e^{is\Delta}$ and $e^{t\Delta}$
commute with tensor products. Therefore, the claimed monotonicity in
the $(1,8,4)$ case follows from the corresponding claim in the
$(2,4,4)$ case. This completes the proof of Theorem \ref{t:main}.
\end{proof}

It is transparent from the proof of Theorem \ref{t:main} and
\eqref{e:csform} how one may obtain an explicit formula for
$Q_{p,q}'(t)$ provided $q$ is an even integer which divides $p$ and
$f$ is sufficiently well-behaved (such as bounded with compact
support). For example, using the notation used in the above proof of
Theorem \ref{t:main},
\begin{equation*}
  \frac{d}{dt}(Q_{6,6}(t)^6) = \frac{1}{8\sqrt{3}} \int_O \int_{\mathbb{R}^3} |V(t,X) - \rho^tV(t,\rho X)|^2 (e^{t
  \Delta}|F|^2)^{1/2}(e^{t
  \Delta}|F_\rho|^2)^{1/2}\,dXd\mathcal{H}(\rho)
\end{equation*}
where $V(t,\cdot)$ denotes the time dependent vector field on
$\mathbb{R}^3$ given by
\begin{equation*}
V(t,X) = \nabla(\log e^{t \Delta}|F|^2)(X)
\end{equation*}
and $\rho^t$ denotes the transpose of $\rho$.

\section{Further results}

\subsection{Mehler-flow}
The operator $L := \Delta - \langle x,\nabla \rangle$ generates the
Mehler semigroup $e^{tL}$ (sometimes called the Ornstein--Uhlenbeck
semigroup) given by
\begin{equation*} \label{e:mehlerdefn}
  e^{tL}f(x) = \int_{\mathbb{R}^d} f(e^{-t}x+\sqrt{1-e^{-2t}}y)\,d\gamma_d(y)
\end{equation*}
for suitable functions $f$ on $\mathbb{R}^d$, where $d\gamma_d$ is
the gaussian probability measure on $\mathbb{R}^d$ given by
\begin{equation*}
d\gamma_d(y) = \frac{1}{(2\pi)^{d/2}} e^{-|y|^2/2}dy.
\end{equation*}
Naturally, $u(t,\cdot) := e^{tL}f$ satisfies the evolution equation
\begin{equation*} \label{e:mehlereqn}
  \partial_t u = Lu
\end{equation*}
with initial datum $u(0,x) = f(x)$. It will be convenient to
restrict our attention to functions $f$ which are bounded and
compactly supported.

The purpose of this remark is to highlight that when $(d,p,q)$ is
one of $(1,6,6)$, $(1,8,4)$ or $(2,4,4)$ the Strichartz norm also
exhibits a certain monotonicity subject to the input evolving
according to a quadratic Mehler-flow.

\begin{theorem}\label{t:OU} Suppose $f$ is a bounded and
compactly supported function on $\mathbb{R}^d$. If $(d,p,q)$ is
Schr\"odinger admissible and $q$ is an even integer which divides
$p$ then the quantity
\begin{equation*}
  Q(t) := \|e^{is\Delta}(e^{-\frac{1}{2}|\cdot|^2}e^{tL}|f|^2)^{1/2}
  \|_{L^p_sL^q_x(\mathbb{R} \times \mathbb{R}^d)}
\end{equation*}
is nondecreasing for all $t>0$.
\end{theorem}

As a consequence of Theorem \ref{t:OU}, we may again recover sharp
forms of the Strichartz estimates in \eqref{e:Strichartz} for such
exponents by considering the limiting values of $Q(t)$ as $t$
approaches zero and infinity. In particular, since
\begin{equation*}
  e^{tL}|f|^2(x) = \int_{\mathbb{R}^d} |f|^2(e^{-t}x+\sqrt{1-e^{-2t}}y)\,d\gamma_d(y)
\end{equation*}
it follows that, for each $x \in \mathbb{R}^d$, $e^{tL}|f|^2(x)$
tends to $\int_{\mathbb{R}^d} |f|^2\,d\gamma_d$ as $t$ tends to
infinity. Thus, the monotonicity of $Q$ implies that
\begin{equation*}
    \| e^{is \Delta}(e^{-\frac{1}{4}|\cdot|^2}|f|)\|_{L^p_sL^q_x(\mathbb{R} \times \mathbb{R}^d)}
   \leq
   \| e^{is \Delta}(e^{-\frac{1}{4}|\cdot|^2})\|_{L^p_sL^q_x(\mathbb{R} \times \mathbb{R}^d)}
   \bigg( \int_{\mathbb{R}^d} |f|^2\,d\gamma_d \bigg)^{1/2}
\end{equation*}
for each bounded and compactly supported function $f$ on
$\mathbb{R}^d$. Thus,
\begin{equation*}
  \| e^{is \Delta}g\|_{L^p_sL^q_x(\mathbb{R} \times \mathbb{R}^d)} \leq
  \|e^{is \Delta}\big(\tfrac{1}{(2\pi)^{d/2}}e^{-\frac{1}{2}|\cdot|^2}\big)^{1/2}\|_{L^p_sL^q_x(\mathbb{R} \times \mathbb{R}^d)}
\|g\|_{L^2(\mathbb{R}^d)}
\end{equation*}
for each $g \in L^2(\mathbb{R}^d)$.

The first key ingredient in the proof of Theorem \ref{t:OU} is to
observe that an analogue of Lemma \ref{l:CS} holds for Mehler-flow.

\begin{lemma} \label{l:OUmon} Let $n \in \mathbb{N}$ and let $f_1$ and $f_2$ be
nonnegative, bounded and compactly supported functions on
$\mathbb{R}^n$. Then the quantity
\begin{equation*}
    \Lambda(t) := \int_{\mathbb{R}^n} (e^{-\frac{1}{2}|\cdot|^2}e^{tL}f_1)^{1/2}(e^{-\frac{1}{2}|\cdot|^2}e^{tL}f_2)^{1/2}
\end{equation*}
is nondecreasing for all $t>0$.
\end{lemma}

\begin{proof} Let $\mathfrak{u}_j : (0,\infty) \times \mathbb{R}^n
\rightarrow \mathbb{R}$ be given by
\begin{equation} \label{e:germanuj}
  \mathfrak{u}_j(t,x) = e^{-\frac{1}{2}|x|^2}e^{tL}f_j(x) = e^{-\frac{1}{2}|x|^2}\int_{\mathbb{R}^n} f_j(e^{-t}x+\sqrt{1-e^{-2t}}y)\,d\gamma_n(y)
\end{equation}
for $j=1,2$. It is straightforward to check that
  \begin{equation*}
  \partial_t \mathfrak{u}_j = \Delta \mathfrak{u}_j + \langle x,\nabla \mathfrak{u}_j \rangle +
  n\mathfrak{u}_j
\end{equation*}
and furthermore
\begin{equation*}
  \partial_t(\log \mathfrak{u}_j) = \text{div}(v_j) +
  |v_j|^2 + \langle x,v_j \rangle + n,
\end{equation*}
where $v_j := \nabla(\log \mathfrak{u}_j)$. Therefore,
\begin{equation*}
  \Lambda'(t) = I + II
\end{equation*}
where
\begin{equation*}
  I := \frac{1}{2}\int_{\mathbb{R}^n} (\text{div}(v_1)
  + \text{div}(v_2) + |v_1|^2 +
  |v_2|^2)(t,x)\,\mathfrak{u}_1(t,x)^{1/2}\mathfrak{u}_2(t,x)^{1/2}\,dx
\end{equation*}
and
\begin{equation*}
  II := \int_{\mathbb{R}^n} (\langle x,\tfrac{1}{2}v_1 + \tfrac{1}{2}v_2
  \rangle + n)(t,x)\,\mathfrak{u}_1(t,x)^{1/2}\mathfrak{u}_2(t,x)^{1/2}\,dx.
\end{equation*}
Since $f_j$ is bounded with compact support it follows from the
explicit formula for $\mathfrak{u}_j$ in \eqref{e:germanuj} that
$v_j(t,x)$ grows at most polynomially in $x$ for each fixed $t
> 0$ and consequently $\int_{\mathbb{R}^n}
\text{div}(\mathfrak{u}_1^{1/2}\mathfrak{u}_2^{1/2}v_j)$ vanishes by
the divergence theorem. It follows that
\begin{equation*}
  I = \frac{1}{4} \int_{\mathbb{R}^n}
  |v_1(t,x)-v_2(t,x)|^2\mathfrak{u}_1(t,x)^{1/2}\mathfrak{u}_2(t,x)^{1/2}\,dx,
\end{equation*}
which is manifestly nonnegative. Since
\begin{equation*}
  II = \int_{\mathbb{R}^n}
  \text{div}(\mathfrak{u}_1(t,x)^{1/2}\mathfrak{u}_2(t,x)^{1/2}x)\,dx
\end{equation*}
we can again appeal to the divergence theorem to deduce that $II$
vanishes. Hence $\Lambda'(t)$ is nonnegative for each $t > 0$.
\end{proof}

The argument in the above proof of Lemma \ref{l:OUmon} is very much
in the spirit of the heat-flow monotonicity results in \cite{CLL}
and \cite{BCCT} and naturally extends to the setting of the
geometric Brascamp--Lieb inequality. In particular, for
$j=1,\ldots,m$ suppose that $p_j \geq 1$ and $B_j : \mathbb{R}^n
\rightarrow \mathbb{R}^{n_j}$ is a linear mapping such that
$B_j^*B_j$ is a projection and $\sum_{j=1}^m \tfrac{1}{p_j}B_j^*B_j
= I_{\mathbb{R}^n}$. Then the quantity
\begin{equation*}
  \int_{\mathbb{R}^n} \prod_{j=1}^m
  (e^{-\frac{1}{2}|B_jx|^2}(e^{tL}f_j)(B_jx))^{1/p_j}\,dx = (2\pi)^{d/2}\int_{\mathbb{R}^n} \prod_{j=1}^m
  (e^{tL}f_j)(B_jx)^{1/p_j}\,d\gamma_n(x)
\end{equation*}
is nondecreasing for each $t>0$ provided each $f_j$ is a
nonnegative, bounded and compactly supported function on
$\mathbb{R}^{n_j}$. This is due to Barthe and Cordero-Erausquin
\cite{BC} in the case where each $B_j$ has rank one. A modification
of the argument gives the general rank case (see \cite{CL} for
closely related results).

By following the same argument employed in our proof of Theorem
\ref{t:main}, to conclude the proof of Theorem \ref{t:OU} it
suffices to note that Mehler-flow appropriately respects tensor
products and isometries. In particular we need that if $F$ is the
$m$-fold tensor product of $f$ then
\begin{equation} \label{e:mehlertensor}
  \bigotimes_{j=1}^m e^{-\frac{1}{2}|\cdot|^2}e^{tL}|f|^2 =
  e^{-\frac{1}{2}|\cdot|^2}e^{tL}|F|^2
\end{equation}
and, for each isometry $\rho$ on $(\mathbb{R}^d)^m$,
\begin{equation} \label{e:mehlerisom}
  \bigotimes_{j=1}^m e^{-\frac{1}{2}|\cdot|^2}e^{tL}|f|^2(\rho \,\cdot) =
  e^{-\frac{1}{2}|\cdot|^2}e^{tL}|F_\rho|^2
\end{equation}
where $F_\rho := F(\rho \,\cdot)$. Here, the operators $|\cdot|$ and
$L$ are acting on the number of variables dictated by context. The
verification of \eqref{e:mehlertensor} and \eqref{e:mehlerisom} is
an easy exercise.

\subsection{Mitigating powers of $t$} It is possible to relax the
quadratic nature of the heat-flow in the quantity $Q_{p,q}$ in
Theorem \ref{t:main} by inserting a mitigating factor which is a
well-chosen power of $t$.
\begin{theorem} \label{t:power}
  Suppose that $(p,q,d)$ is Schr\"{o}dinger admissible and $q$ is an
  even integer which divides $p$. If $f$ is a nonnegative integrable function on $\mathbb{R}^d$
  and $\alpha \in [1/2,1]$ then the quantity
  \begin{equation*}
     t^{d(\alpha - 1/2)/2}\|e^{is \Delta}(e^{t\Delta} f)^{\alpha}\|_{L^p_sL^q_x(\mathbb{R} \times \mathbb{R}^d)}.
  \end{equation*}
  is nondecreasing for each $t > 0$.
\end{theorem}
By \cite{BCCT}, we have that Lemma \ref{l:CS} generalises to the
statement that
\begin{equation} \label{e:CSgen}
     t^{n(\alpha-1/2)}\int_{\mathbb{R}^n} (e^{t\Delta}f_1)^{\alpha}(e^{t\Delta}f_2)^{\alpha}
\end{equation}
is nondecreasing for all $t>0$ provided $n \in \mathbb{N}$, $\alpha
\in [1/2,1]$ and $f_1, f_2$ are nonnegative integrable functions on
$\mathbb{R}^n$. Thus Theorem \ref{t:power} follows by the same
argument in our proof of Theorem \ref{t:main}.

\subsection{Higher dimensions}
Theorem \ref{t:main} raises obvious questions about higher
dimensional analogues and consequently the potential of our approach
to prove the sharp form of \eqref{e:Strichartz} in all dimensions
(at least for nonnegative initial data $f$). Recently, Shao
\cite{shao} has shown that for non-endpoint Schr\"odinger admissible
triples $(p,q,d)$,
   \begin{equation*} \sup \{ \|e^{is\Delta}f\|_{L^p_sL^q_{x}(\mathbb{R} \times \mathbb{R}^d)} : \|f\|_{L^2(\mathbb{R}^d)} = 1\}
   \end{equation*}
is at least attained, although does not determine the explicit form
of an extremiser. There is some anecdotal evidence in \cite{BBC} to
suggest that Theorem \ref{t:main} may not extend to all
Schr\"{o}dinger admissible triples $(d,p,q)$. Nevertheless, we end
this section with a discussion of some results in this direction
which we believe to be of some interest.

We shall consider the case $p=q=2+4/d$ and it will be convenient to
denote this number by $p(d)$. Since $p(d)$ is not an even integer
for $d\geq 3$, one possible approach to the question of monotonicity
of $Q_{p(d),p(d)}$, given by \eqref{e:Qpq}, is to attempt to embed
the Strichartz norm
\begin{equation*}
|||f|||_{p(d)} := \|e^{is\Delta}f \|_{L^{2+4/d}_{s,x}(\mathbb{R}
\times \mathbb{R}^d)}
\end{equation*}
in a one-parameter family of norms $|||\cdot|||_{p}$ which are
appropriately monotone under a quadratic flow for $p\in
2\mathbb{N}$, and for which the resulting monotonicity formula may
be ``extrapolated", in a sign preserving way, to $p=p(d)$. Such an
approach has proved effective in the context of the general
Brascamp--Lieb inequalities, and was central to the approach to the
multilinear Kakeya and Strichartz inequalities in \cite{BCT}.

Our analysis for $d=1,2$ suggests (albeit rather indirectly) a
natural candidate for such a family of norms. For each $d \in
\mathbb{N}$ and $p > p(d)$, we define a norm
 $|||\cdot |||_p$ on $\mathcal{S}(\mathbb{R}^d)$ by
 \begin{equation*}
   ||| f |||^p_p = \frac{(p(d)/\pi)^{d/2}}{(2\pi)^{d+2}}
   \int_{\mathbb{R}^d} \int_{\mathbb{R}^d} \int_0^\infty
   \int_\mathbb{R} \left|\int_{\mathbb{R}^d}
     e^{-|z-\sqrt{\zeta}\xi|^2}
    e^{i(x \cdot \xi - s|\xi|^2)}\widehat{f}(\xi)\,d\xi\right|^p
   \frac{\zeta^{\nu-1}}{\Gamma(\nu)}\,ds d\zeta
   dzdx,
  \end{equation*}
where $\nu = d(p-p(d))/4$. For $|||\cdot|||_p$ we have the
following.
\begin{theorem}\label{t:modified}
As $p$ tends to $p(d)$ the norm $|||f|||_p$ converges to the
Strichartz norm $\|e^{is\Delta}f\|_{L^{p(d)}_{s,x}}$ for each $f$
belonging to the Schwartz class on $\mathbb{R}^d$. Additionally, if
$\alpha \in [1/2,1]$ and $f$ is a nonnegative integrable function on
$\mathbb{R}^d$ then
\begin{equation*}
\widetilde{Q}_{\alpha,p}(t) :=
t^{d(\alpha-1/2)/2}|||(e^{t\Delta}f)^{\alpha}|||_p
\end{equation*}
is nondecreasing for all $t>0$ whenever $p$ is an even integer.
\end{theorem}

 \begin{rems} \emph{(1) This ``modified Strichartz norm" $|||f|||_p$  is related in spirit
to the norm
\begin{equation*}
\|I_\beta \;e^{is\Delta}f\|_{L^{p}_{s,x}(\mathbb{R} \times
\mathbb{R}^d)},
\end{equation*}
where $I_\beta$ denotes the fractional integral of order
$\beta=d(p-p(d))/2p$. Although it is true that for all $p\geq p(d)$,
\begin{equation*}
\|I_\beta \;e^{is\Delta}f\|_{L^{p}_{s,x}(\mathbb{R} \times
\mathbb{R}^d)} \leq C\|f\|_{L^2(\mathbb{R}^d)}
\end{equation*}
for some finite constant $C$, the desired heat-flow monotonicity for
$p\in 2\mathbb{N}$ is far from apparent for these norms.}

\emph{(2) Both the Strichartz norm and the modified Strichartz norms
$|||\cdot|||_p$ are invariant under the Fourier transform; that is
\begin{equation}\label{invstr}
\|e^{is\Delta}\widehat{f}\|_{L^{p(d)}_{s,x}(\mathbb{R} \times
\mathbb{R}^d)} = \|e^{is\Delta}f\|_{L^{p(d)}_{s,x}(\mathbb{R} \times
\mathbb{R}^d)}
\end{equation}
for all $d\in\mathbb{N}$ and
\begin{equation}\label{e:invmod}
|||\widehat{f}|||_p=|||f|||_p
\end{equation}
for all $p>p(d)$ and $d\in\mathbb{N}$. This observation follows by
direct computation and simple changes of variables; for the
Strichartz norm it was noted for $d=1,2$ in \cite{HZ}. We note that
in the proof of Theorem \ref{t:modified} below we use the invariance
in \eqref{e:invmod} for even integers $p$ which (as we will see)
follows from Parseval's theorem.}

\emph{(3)  For every integer $m \geq 2$ and in all dimensions $d
\geq 1$, a corollary to the case $\alpha=1/2$ of Theorem
\ref{t:modified} is the following sharp inequality,
  \begin{equation*}
    |||f|||_{2m} \leq C_{d,m}\|f\|_{L^2(\mathbb{R}^d)},
  \end{equation*}
  where the constant $C_{d,m}$ is given by
  \begin{equation} \label{e:sharpmod}
    C_{d,m}^{2m} = \frac{\pi^\nu}{2^{\nu+1}m^d\Gamma(\nu +1)}\left(\frac{p(d)}{2} \right)^{d/2}.
  \end{equation}
  Here $\nu = d(2m-p(d))/4$ as before.}

\emph{(4)  It is known that for nonnegative integrable functions $f$
on $\mathbb{R}^d$ the quantity
  \begin{equation*}
  \|\widehat{(e^{t\Delta}f)^{1/p}}\|_{L^{p'}(\mathbb{R}^d)}
  \end{equation*}
  is nondecreasing for each $t > 0$ provided the conjugate exponent
   $p'$ is an even integer; this follows from \cite{BCCT} and \cite{BB}. However,
   tying in with our earlier comment on the extension of Theorem
   \ref{t:main} to all admissible Schr\"odinger admissible
   exponents, in \cite{BBC} we show that whenever $p' > 2$ is not an even integer
   there exists a nonnegative integrable function $f$ such that $Q(t)$
   is \emph{strictly decreasing} for all sufficiently small $t>0$.}
\end{rems}

\begin{proof}[Proof of Theorem \ref{t:modified}]
To see the claimed limiting behaviour of $|||f|||_p$ as $p$ tends to
$p(d)$ observe that
\begin{equation} \label{e:limit}
  \lim_{\nu \rightarrow 0} \frac{1}{\Gamma(\nu)} \int_0^\infty \phi(\nu,\zeta)
  \zeta^{\nu-1}\, d\zeta = \phi(0,0)
\end{equation}
for any $\phi$ on $[0,\infty) \times [0,\infty)$ satisfying certain
mild regularity conditions. For example, \eqref{e:limit} holds if
$\phi$ is continuous at the origin and there exist constants
$C,\varepsilon > 0$ such that, locally uniformly in $\nu$, one has
$|\phi(\nu,\zeta) - \phi(\nu,0)| \leq C|\zeta|^\varepsilon$ for all
$\zeta$ in a neighbourhood of zero and $|\phi(\nu,\zeta)| \leq
C|\zeta|^{-\varepsilon}$ for all $\zeta$ bounded away from a
neighbourhood of zero. One can check that standard estimates (for
example, Strichartz estimates of the form \eqref{e:Strichartz} for
compactly supported functions) imply that for $f$ belonging to the
Schwartz class on $\mathbb{R}^d$,
\begin{equation*}
  \phi(\nu,\zeta) = \int_{\mathbb{R}^d} \int_{\mathbb{R}^d} \int_\mathbb{R} \left|\int_{\mathbb{R}^d}
  e^{-|z-\sqrt{\zeta}\xi|^2} e^{i(x \cdot \xi - s|\xi|^2)}
  \widehat{f}(\xi)\,d\xi\right|^p \,dsdxdz
\end{equation*}
satisfies such conditions.

We now turn to the monotonicity claim, beginning with some notation.
Suppose that $p=2m$ for some positive integer $m$. For a nonnegative
$f\in\mathcal{S}(\mathbb{R}^d)$ let
$F:\mathbb{R}^{md}\rightarrow\mathbb{R}$ be given by $F(X) =
\otimes_{j=1}^m f(X)$ where $X = (\xi_1,\ldots,\xi_m) \in
(\mathbb{R}^d)^m \cong \mathbb{R}^{md}$. Next we define the subspace
$W$ of $\mathbb{R}^{md}$ to be the linear span of
$\one_1,\ldots,\one_m$ where for each $1 \leq j \leq d$, $\one_j :=
(e_j,\ldots,e_j)/\sqrt{m}$ and $e_j$ denotes the $j$th standard
basis vector of $\mathbb{R}^d$. For a vector $X\in\mathbb{R}^{md}$
we denote by $X_W$ and $X_{W^\perp}$ the orthogonal projections of
$X$ onto $W$ and $W^\perp$ respectively. Now,
\begin{equation*}
|||f|||_{2m}^{2m} = \tfrac{1}{2^{d+1}\pi} \big(\tfrac{p(d)}{m\pi}
\big)^{d/2} \int
\delta(X_W-Y_W)\delta(|X|^2-|Y|^2)K(X,Y)F(X)F(Y)\,dXdY,
\end{equation*}
where we integrate over $\mathbb{R}^{md} \times \mathbb{R}^{md}$ and
\begin{eqnarray*}
\begin{aligned}
K(X,Y) &=
\int_0^\infty\frac{\zeta^{\nu-1}}{\Gamma(\nu)}e^{-\zeta(|X|^2+|Y|^2)}\int_{\mathbb{R}^d}
e^{\sqrt{m\zeta}z\cdot(X_W+Y_W)}e^{-\frac{m}{2}|z|^2}\,dzd\zeta\\
&= \big( \tfrac{2\pi}{m}
\big)^{d/2}\int_0^\infty\frac{\zeta^{\nu-1}}{\Gamma(\nu)}e^{-\zeta(|X|^2+|Y|^2)}e^{\frac{1}{2}\zeta|X_W+Y_W|^2}\,d\zeta
\end{aligned}
\end{eqnarray*}
for $(X,Y) \in \mathbb{R}^{md} \times \mathbb{R}^{md}$. Thus, on the
support of the delta distributions ($X_W=Y_W$ and $|X|^2=|Y|^2$) we
have
\begin{eqnarray*} \label{lapl}
\begin{aligned}
K(X,Y) &= \big( \tfrac{2\pi}{m} \big)^{d/2}
\int_0^\infty\frac{\zeta^{\nu-1}}{\Gamma(\nu)}e^{-2\zeta(|X|^2-|X_W|^2)}\,d\zeta
\\ &= \tfrac{1}{2^\nu}\big( \tfrac{2\pi}{m}
\big)^{d/2} \frac{1}{(|X|^2-|X_W|^2)^\nu} = \tfrac{1}{2^\nu}\big(
\tfrac{2\pi}{m} \big)^{d/2}\frac{1}{|X_{W^{\perp}}|^{2\nu}}.
\end{aligned}
\end{eqnarray*}
Therefore
\begin{equation} \label{e:modrep}
|||f|||_{2m}^{2m} = \tfrac{\pi^\nu}{2^{\nu+1}m^d\Gamma(\nu
+1)}(\tfrac{p(d)}{2})^{d/2} \int_{\mathbb{R}^{md}}F(X)PF(X)\,dX,
\end{equation}
where $P$ is given by
\begin{equation*}
PF(X)= \tfrac{\Gamma(\nu+1)}{\pi^{\nu+1}}
\frac{1}{|X_{W^\perp}|^{2\nu}}\int_{\mathbb{R}^{md}}\delta(X_W-Y_W)\delta(|X|^2-|Y|^2)F(Y)\,dY.
\end{equation*}
Using polar coordinates in $W^\perp$ in the above integral and
recalling that $\nu = d(2m-p(d))/4$ identifies $P$ as the orthogonal
projection onto functions on $\mathbb{R}^{md}$ which are invariant
under the action of $O$, the group of isometries on
$\mathbb{R}^{md}$ which coincide with the identity on $W$; i.e.
\begin{equation*}
PF(X)=\int_{O}F(\rho X)\,d\mathcal{H}(\rho),
\end{equation*}
where $d\mathcal{H}$ denotes the right-invariant Haar probability
measure on $O$.

Finally, applying the representation of $|||f|||_{2m}^{2m}$ in
\eqref{e:modrep} to the quantity $\widetilde{Q}_{\alpha,2m}$, and
appealing to the nondecreasingness of the quantity in
\eqref{e:CSgen},  we conclude that $\widetilde{Q}_{\alpha,2m}(t)$ is
nondecreasing for all $t > 0$ and all $\alpha \in [1/2,1]$. This
completes the proof of Theorem \ref{t:modified}.
\end{proof}

\bibliographystyle{amsalpha}

\end{document}